\documentclass[letterpaper, 10 pt, conference]{ieeeconf}  

\IEEEoverridecommandlockouts                              
\usepackage{amsmath,amsfonts,amssymb,color}
\usepackage{epsfig}
\usepackage{psfrag}
\usepackage{algorithm}
\usepackage{algpseudocode}
\usepackage{array}
\usepackage{epstopdf}
\usepackage{cite}
\usepackage{footmisc}
\usepackage{mathtools}
\DeclarePairedDelimiter{\ceil}{\lceil}{\rceil}

\newtheorem{problem}{Problem}
\newtheorem{theorem}{Theorem}
\newtheorem{corollary}{Corollary}
\newtheorem{lemma}{Lemma}
\newtheorem{remark}{Remark}
\newtheorem{definition}{Definition}
\newtheorem{proposition}{Proposition}
\newtheorem{example}{Example}

\renewcommand{\theenumi}{(\alph{enumi}}

\title{\LARGE \bf
On the Complexity and Approximability of Optimal Sensor Selection for Kalman Filtering
}

\author{Lintao Ye, Sandip Roy and Shreyas Sundaram
\thanks{This research was supported by NSF grant CMMI-1635014. Lintao Ye and Shreyas Sundaram are with the School of Electrical
and Computer Engineering at Purdue University. Email: \{ye159,sundara2\}@purdue.edu.  Sandip Roy is with the School of Electrical and Computer Engineering and Computer Science at Washington State University.  Email: sroy@eecs.wsu.edu.
}
}

\begin{document}

\maketitle
\thispagestyle{empty}
\pagestyle{empty}

\begin{abstract}
Given a linear dynamical system, we consider the problem of selecting (at design-time) an optimal set of sensors (subject to certain budget constraints) to minimize the trace of the steady state error covariance matrix of the Kalman filter. Previous work has shown that this problem is NP-hard for certain classes of systems and sensor costs; in this paper, we show that the problem remains NP-hard even for the special case where the system is stable and all sensor costs are identical.  Furthermore, we show the stronger result that there is no constant-factor (polynomial-time) approximation algorithm for this problem. This contrasts with other classes of sensor selection problems studied in the literature, which typically pursue constant-factor approximations by leveraging greedy algorithms and submodularity of the cost function. Here, we provide a specific example showing that greedy algorithms can perform arbitrarily poorly for the problem of design-time sensor selection for Kalman filtering.
\end{abstract}

\section{Introduction}

Selecting an appropriate set of actuators or sensors in order to achieve certain performance requirements is an important problem in control system design (e.g., \cite{van2001review}, \cite{olshevsky2014minimal}, \cite{liu2003problem}). For instance, in the case of linear Gauss-Markov systems, researchers have studied techniques to select sensors dynamically (at run-time) or statically (at design-time) in order to minimize certain metrics of the error covariance of the corresponding Kalman Filter. These  are known as sensor scheduling problems (e.g., \cite{vitus2012efficient}, \cite{huber2012optimal}, \cite{jawaid2015submodularity}) and design-time sensor selection problems (e.g., \cite{chmielewski2002theory}, \cite{dhingra2014admm}, \cite{tzoumas2016sensor}, \cite{zhang2017sensor}), respectively. These problems are NP-hard in general (e.g., \cite{zhang2017sensor}), and various approximation algorithms have been proposed to solve them. For example, the concept of submodularity \cite{nemhauser1978analysis} has been widely used to analyze the performance of greedy algorithms for sensor scheduling and selection (e.g., \cite{krause2008near}, \cite{shamaiah2010greedy}, \cite{jawaid2015submodularity}, \cite{summers2016submodularity}).

In this paper, we consider the design-time sensor selection problem for optimal filtering of discrete-time linear dynamical systems. We study the problem of choosing a subset of sensors (under given budget constraints) to optimize the steady state error covariance of the corresponding Kalman filter.  We refer to this problem as the {\it Kalman filtering sensor selection (KFSS)} problem. We summarize some related work as follows.

In \cite{chmielewski2002theory}, the authors considered the design-time sensor selection problem of a sensor network for discrete-time linear dynamical systems, also known as dynamic data-reconciliation problems. They assumed that each sensor measures one component of the system state and the measured and unmeasured states are related via network defined mass-balance functions. The objective is to minimize the  cost of implementing the network configuration subject to certain performance criteria. They transformed the problem into convex optimization problems, but did not give complexity analysis of the problem. In contrast, we consider the problem of minimizing the estimation error under a cardinality constraint on the chosen sensors without network configuration and analyze the complexity of the problem.

In \cite{tzoumas2016sensor}, the authors studied the design-time sensor selection problem for discrete-time linear time-varying systems over a finite time horizon, under the assumption that each sensor measures one component of the system state vector. The objective is to minimize the number of chosen sensors while guaranteeing a certain level of performance (or alternatively, to minimize the estimation error with a cardinality constraint on the chosen sensors). In contrast, we consider general measurement matrices and aim to minimize the steady state estimation error.  

The papers \cite{yang2015deterministic} and \cite{zhang2017sensor} considered the same design-time sensor selection as the one we consider here. In \cite{yang2015deterministic}, the authors expressed the problem as a semidefinite program (SDP). However, they did not provide theoretical guarantees on the performance of the proposed algorithm. The paper \cite{zhang2017sensor} showed that the problem is NP-hard and gave examples showing that the cost function is not submodular in general. The authors also provided upper bounds on the performance of algorithms for the problem; these upper bounds were functions of the system matrices. Although \cite{zhang2017sensor} showed via simulations that greedy algorithms performed well for several randomly generated systems, the question of whether such algorithms (or other polynomial-time algorithms) could provide constant-factor approximation ratios for the problem was left open.

Our contributions to this problem are as follows. First, we show that the KFSS problem is NP-hard {\it even} for the special case when the system is stable and all sensors have the same cost. This complements and strengthens the complexity result in \cite{zhang2017sensor}, which only showed NP-hardness for two subclasses of problem instances: (1) when the system is unstable and the sensor costs are identical, and (2) when the system is stable but the sensor costs are arbitrary.  The NP-hardness of those cases followed in a relatively straightforward manner via reductions from the minimal controllability \cite{olshevsky2014minimal} and knapsack \cite{garey1979computers} problems, respectively.  In contrast, the stronger NP-hardness proof that we provide here requires a more careful analysis, and makes connections to finding sparse solutions to linear systems of equations, yielding new insights into the problem.

After establishing NP-hardness of the problem as above, our second (and most significant) contribution is to show that {\it there is no constant factor approximation algorithm for this problem} (unless $P = NP$).  In other words, there is no polynomial-time algorithm that can find a sensor selection that is always guaranteed to yield a mean square estimation error (MSEE) that is within any constant finite factor of the MSEE for the optimal selection.  This stands in stark contrast to other sensor selection problems studied in the literature, which leveraged submodularity of their associated cost functions to provide greedy algorithms with constant-factor approximation ratios \cite{tzoumas2016sensor}.

Our inapproximability result above immediately implies that greedy algorithms cannot provide constant-factor guarantees for our problem.  Our third contribution in this paper is to explicitly show how greedy algorithms can provide arbitrarily poor performance even for very small instances of the KFSS problem (i.e., in systems with only three states and three sensors to choose from).  

The rest of this paper is organized as follows. In Section \ref{sec:problem formulation}, we formulate the KFSS problem. In Section \ref{sec:complexity analysis}, we analyze the complexity of the KFSS problem. In Section \ref{sec:greedy example}, we study a greedy algorithm for the KFSS problem and analyze its performance. In Section \ref{sec:conclusion}, we conclude the paper.

\subsection{Notation and terminology}
The set of natural numbers, integers, real numbers, rational numbers, and complex numbers are denoted as $\mathbb{N}$, $\mathbb{Z}$, $\mathbb{R}$, $\mathbb{Q}$ and $\mathbb{C}$, respectively. For any $x\in\mathbb{R}$, denote $\ceil{x}$ as the least integer greater than or equal to $x$. For a square matrix $P\in\mathbb{R}^{n\times n}$, let $P^T$, $\textrm{trace}(P)$, $\textrm{det}(P)$, $\{\lambda_i(P)\}$ and $\{\sigma_i(P)\}$ be its transpose, determinant, set of eigenvalues and set of singular values, respectively. The set of eigenvalues $\{\lambda_i(P)\}$ of $P$ are ordered with nondecreasing magnitude, i.e., $|\lambda_1(P)|\ge\dots\ge|\lambda_n(P)|$; the same order applies to the set of singular values $\{\sigma_i(P)\}$. Denote $P_{ij}$ as the element in the  $i$th row and $j$th column of $P$.  A positive definite (resp. positive semi-definite) matrix $P$ is denoted as $P\succ0$ (resp. $P\succeq0$), and $P\succeq Q$ if $P-Q\succeq0$. The set of $n$ by $n$ positive definite (resp. positive semi-definite) matrices is denoted as $\mathbb{S}_{++}^n$ (resp. $\mathbb{S}_{+}^n$). The identity matrix with dimension $n$ is denoted as $I_{n\times n}$. For a vector $v$, denote $v_i$  as the $i$th element of $v$ and let $\textrm{supp}(v)$ be its support, where $\textrm{supp}(v)=\{i:v_i\neq 0\}$.  Denote the Euclidean norm of $v$ by $\lVert v\lVert_2$. Define $\mathbf{e}_i$ to be a row vector where the $i$th element is $1$ and all the other elements are zero; the dimension of the vector can be inferred from the context. For a random variable $\omega$, let $\mathbb{E}(\omega)$ be its expectation. For a set $\mathcal{A}$, let $|\mathcal{A}|$ be its cardinality.

\section{Problem Formulation} \label{sec:problem formulation}
Consider the discrete-time linear system
\begin{equation}
x[k+1] = Ax[k] + w[k],
\label{eqn:system dynamics}
\end{equation}
where $x[k]\in\mathbb{R}^n$ is the system state, $w[k]\in\mathbb{R}^n$ is a zero-mean white Gaussian noise process with $\mathbb{E}[w[k](w[k])^T]=W$ for all $k\in\mathbb{N}$, and $A\in\mathbb{R}^{n\times n}$ is the system dynamics matrix. We assume throughout this paper that the pair $(A,W^{\frac{1}{2}})$ is stabilizable.

Consider a set $\mathcal{Q}$ consisting of $q$ sensors. Each sensor $i\in\mathcal{Q}$ provides a measurement of the system in the form
\begin{equation}
\label{eqn:single sensor measurement}
y_i[k]=C_ix[k]+v_i[k],
\end{equation}
where $C_i\in\mathbb{R}^{s_i\times n}$ is the state measurement matrix for sensor $i$, and $v_i[k]\in\mathbb{R}^{s_i}$ is a zero-mean white Gaussian noise process. We further define $y[k]\triangleq\big[(y_1[k])^T\ \cdots \ (y_q[k])^T\big]^T$, $C\triangleq\big[C_1^T\ \cdots\ C_q^T\big]^T$ and $v[k]\triangleq\big[(v_1[k])^T\ \cdots \ (v_q[k])^T\big]^T$. Thus, the output provided by all sensors together is given by
\begin{equation}
\label{eqn:all sensors measurements}
y[k]=C x[k]+v[k],
\end{equation}
where $C\in\mathbb{R}^{s\times n}$ and $s=\sum_{i=1}^q s_i$. We denote $\mathbb{E}[v[k](v[k])^T]=V$ and consider $\mathbb{E}[v[k](w[j])^T]=\mathbf{0}$, $\forall k, j\in\mathbb{N}$. 

Consider that there are no sensors initially deployed on the system. Instead, the system designer must select a subset of sensors from the set $\mathcal{Q}$ to install. Each sensor $i\in\mathcal{Q}$ has a cost $b_{i}\in\mathbb{R}_{\ge0}$; define the cost vector $b\triangleq \left[\begin{matrix}b_1 & \cdots & b_q\end{matrix}\right]^T$. The designer has a budget $B\in\mathbb{R}_{\ge0}$, representing the total cost that can be spent on sensors from $\mathcal{Q}$.

After a set of sensors is selected and installed, the Kalman filter is then applied to provide an optimal estimate of the states using the measurements from the installed sensors (in the sense of minimizing the MSEE). We define a vector $\mu\in\{0,1\}^q$ as the indicator vector of the selected sensors, where $\mu_i=1$ if and only if sensor $i\in\mathcal{Q}$ is installed. Denote $C(\mu)$ as the measurement matrix of the installed sensors indicated by $\mu$, i.e., $C(\mu)\triangleq\left[\begin{matrix} C_{i_1}^T &  \cdots & C_{i_p}^T \end{matrix}\right]^T$, where $\textrm{supp}(\mu)=\{i_1,\dots,i_p\}$. Similarly, denote $V(\mu)$ as the measurement noise covariance matrix of the installed sensors, i.e., $V(\mu)=\mathbb{E}[\tilde{v}[k](\tilde{v}[k])^T]$, where $\tilde{v}[k]=\big[(v_{i_1}[k])^T\ \cdots \ (v_{i_p}[k])^T\big]^T$. Let $\Sigma_{k|k-1}(\mu)$ and $\Sigma_{k|k}(\mu)$ denote the {\it a priori} error covariance matrix and the 
{\it a posteriori} error covariance matrix of the Kalman filter at time step $k$, respectively, when the sensors indicated by $\mu$ are installed.  We will use the following result \cite{anderson1979optimal}.

\begin{lemma}
Suppose the pair $(A,W^{\frac{1}{2}})$ is stabilizable. For a given indicator vector $\mu$, both $\Sigma_{k|k-1}(\mu)$ and $\Sigma_{k|k}(\mu)$ will converge to finite limits $\Sigma(\mu)$ and $\Sigma^*(
\mu)$, respectively, as $k\to\infty$ if and only if the pair $(A,C(\mu))$ is detectable.
\label{lemma:Anderson optimal filtering}
\end{lemma}

The limit $\Sigma(\mu)$ satisfies the \textit{discrete algebraic Riccati equation (DARE)} \cite{anderson1979optimal}:
\begin{equation}
\label{eqn:DARE}
\begin{split}
&\Sigma(\mu)=A\Sigma(\mu)A^T+W - \\
&A\Sigma(\mu)C(\mu)^T\big(C(\mu)\Sigma(\mu)C(\mu)^T+V(\mu)\big)^{-1}C(\mu)\Sigma(\mu)A^T.
\end{split}
\end{equation}
Applying the matrix inversion lemma \cite{horn1985matrix}, we can rewrite Eq. (\ref{eqn:DARE}) as
\begin{equation}
\label{eqn:DARE after matrix inverse lemma}
\Sigma(\mu)=W+A(\Sigma^{-1}(\mu)+R(\mu))^{-1}A^T,
\end{equation}
where $R(\mu)\triangleq C(\mu)^TV(\mu)^{-1}C(\mu)$ is the sensor information matrix corresponding to sensor selection indicated by $\mu$. Note that the inverses in Eq. $\eqref{eqn:DARE}$ and Eq. $\eqref{eqn:DARE after matrix inverse lemma}$ are interpreted as pseudo-inverses if the arguments are not invertible. For the case when $V=\mathbf{0}$, the matrix inverse lemma does not hold under pseudo-inverse (unless $\mu=\mathbf{0}$), we compute $\Sigma(\mu)$ via Eq. $\eqref{eqn:DARE}$.

The limits $\Sigma(\mu)$ and $\Sigma^*(\mu)$ are coupled as \cite{catlin1989estimation}:
\begin{equation}
\label{eqn:couple priori and posteriori}
\Sigma(\mu)=A\Sigma^*(\mu)A^T+W.
\end{equation}
For the case when the pair $(A,C(\mu))$ is not detectable, we define the limit $\Sigma(\mu)=+\infty$. The Kalman filter sensor selection (KFSS) problem is defined as follows.

\begin{problem}
\label{pro:priori KFSS problem}
(KFSS Problem) Given a system dynamics matrix $A\in\mathbb{R}^{n\times n}$, a measurement matrix $C\in\mathbb{R}^{s\times n}$ containing all of the individual sensor measurement matrices, a system noise covariance matrix $W\in\mathbb{S}_+^n$, a sensor noise covariance matrix $V\in\mathbb{S}_+^{s}$, a cost vector $b\in\mathbb{R}^q_{\ge0}$ and a budget $B\in\mathbb{R}_{\ge0}$, the Kalman filtering sensor selection problem is to find the sensor selection $\mu$, i.e., the indicator vector $\mu$ of the selected sensors, that solves
\begin{equation*}
\begin{split}
&\mathop{\min}_{\mu}\ \text{trace}(\Sigma(\mu))\\
&s.t.\ b^T \mu\le B\\
&\mu\in\{0,1\}^q
\end{split}
\end{equation*}
where $\Sigma(\mu)$ is given by Eq. $(\ref{eqn:DARE})$ if the pair $(A,C(\mu))$ is detectable, and $\Sigma(\mu)=+\infty$, otherwise.
\label{pro: the KFSS problem}
\end{problem}

\section{Complexity Analysis}\label{sec:complexity analysis}
As described in the Introduction, the KFSS problem was shown to be NP-hard in \cite{zhang2017sensor} for two classes of systems and sensor costs.  First, when the $A$ matrix is unstable, the set of chosen sensors must cause the resulting system to be detectable in order to obtain a finite steady state error covariance matrix.  Thus, for systems with unstable $A$ and identical sensor costs, \cite{zhang2017sensor} provided a reduction from the ``minimal controllability'' (or minimal detectability) problem considered in \cite{olshevsky2014minimal} to the KFSS problem.  Second, when the $A$ matrix is stable (so that all sensor selections cause the system to be detectable), \cite{zhang2017sensor} showed that when the sensor costs can be arbitrary, the $0-1$ knapsack problem can be encoded as a special case of the KFSS problem, thereby again showing NP-hardness of the latter problem.

In this section, we provide a stronger result and show that the KFSS problem is NP-hard even for the special case where the $A$ matrix is stable {\it and} all sensors have the same cost.  Hereafter, it will suffice for us to consider the case when $C_i\in\mathbb{R}^{1\times n}$, $\forall i\in\{1,\dots,q\}$, i.e., each sensor corresponds to one row of matrix $C$, and the sensor selection cost vector is $b=[1\ \cdots\ 1]^T$, i.e., each sensor has cost equal to $1$.

We will use the following results in our analysis  (the proofs are provided in the appendix).

\begin{lemma}
Consider a discrete-time linear system as defined in $\eqref{eqn:system dynamics}$ and $\eqref{eqn:all sensors measurements}$. Suppose the system dynamics matrix is of the form $A=\textrm{diag}(\lambda_1,\dots,\lambda_n)$ with $0\le |\lambda_i|<1$, $\forall i \in \{1,\dots,n\}$, the system noise covariance matrix $W$ is diagonal, and the sensor noise covariance matrix is $V=\mathbf{0}$. Then, the following holds for all sensor selections $\mu$.

\begin{enumerate}
\item $\forall i\in\{1,\dots,n\}$, $(\Sigma(\mu))_{ii}$ satisfies 
\begin{equation}
W_{ii}\le (\Sigma(\mu))_{ii}\le \frac{W_{ii}}{1-\lambda_i^2}.
\label{eqn:bound for priori sigma}
\end{equation}
\item If $\exists i\in\{1,\dots,n\}$ such that $W_{ii}=0$, then $(\Sigma(\mu))_{ii}=0$. 

\item If $\exists i\in\{1,\dots,n\}$ such that $\lambda_i=0$, then $(\Sigma(\mu))_{ii}=W_{ii}$. 

\item If $\exists i\in\{1,\dots,n\}$ such that $W_{ii}\neq0$ and the $i$th column of $C(\mu)$ is zero, then $(\Sigma(\mu))_{ii}=\frac{W_{ii}}{1-\lambda_i^2}$.

\item If $\exists i\in\{1,\dots,n\}$ such that $\mathbf{e}_i\in\textrm{rowspace}(C(\mu))$, then $(\Sigma(\mu))_{ii}=W_{ii}$.
\end{enumerate}
\label{Lemma:minimum trace of sigma}
\end{lemma}

\begin{lemma}
Consider a discrete-time linear system as defined in Eq. $(\ref{eqn:system dynamics})$ and Eq. $(\ref{eqn:all sensors measurements})$. Suppose the system dynamics matrix is of the form $A=\textrm{diag}(\lambda_1,0,\dots,0)\in\mathbb{R}^{n\times n}$, where $0<|\lambda_1|<1$, the measurement matrix $C=[1\ \mathbf{\gamma}]$, where $\gamma\in\mathbb{R}^{1\times(n-1)}$, the system noise covariance matrix $W=I_{n\times n}$, and the sensor noise covariance matrix $V=\mathbf{0}$. Then, the MSEE of state $1$, i.e., $\Sigma_{11}$, satisfies 
\begin{equation}
\Sigma_{11}=\frac{1+\alpha^2\lambda_1^2-\alpha^2+\sqrt[]{(\alpha^2-\alpha^2\lambda_1^2-1)^2+4\alpha^2}}{2},
\label{eqn: mean square estimation error of state 1}
\end{equation}
where $\alpha^2\triangleq\lVert\gamma\rVert_2^2$. Moreover, if we view $\Sigma_{11}$ as a function of $\alpha^2$, denoted as $\Sigma_{11}(\alpha^2)$, then $\Sigma_{11}(\alpha^2)$ is a strictly increasing function of $\alpha^2\in\mathbb{R}_{\ge0}$ with $\Sigma_{11}(0)=1$ and ${\lim}_{\alpha\to\infty}\Sigma_{11}(\alpha^2)=\frac{1}{1-\lambda_1^2}$.
\label{lemma:estimation error of state 1}
\end{lemma}

\subsection{NP-hardness of the KFSS problem}

To prove the KFSS problem (Problem $\ref{pro: the KFSS problem}$) is NP-hard, we relate it to the problem described below.

\begin{definition}
$(X3C)$ Given a finite set $D$ with $|D|=3m$ and a collection $\mathcal{C}$ of $3$-element subsets of $D$, an {\it exact cover} for $D$ is a subcollection $\mathcal{C}'\subseteq \mathcal{C}$ such that every element of $D$ occurs in exactly one member of $\mathcal{C}'$.
\end{definition}

\begin{remark}
Since each member in $\mathcal{C}$ is a subset of $D$ with exactly $3$ elements, if there exists an exact cover for $D$, then it must consist of exactly $m$ members of $\mathcal{C}$.
\end{remark}

We will use the following result \cite{garey1979computers}.

\begin{lemma}
Given a finite set $D$ with $|D|=3m$ and a collection $\mathcal{C}$ of $3$-element subsets of $D$, the problem to determine whether $\mathcal{C}$ contains an exact cover for $D$ is NP-complete.
\label{lemma:X3C}
\end{lemma}

We are now in place to prove the following result.

\begin{theorem}
The KFSS problem is NP-hard when the system dynamics matrix $A$ is stable and each sensor $i\in\mathcal{Q}$ has identical cost.
\label{thm: KFSS is NP-hard}
\end{theorem}

\begin{proof}
We give a reduction from $X3C$ to KFSS. Consider an instance of $X3C$ to be a finite set $D$ with $|D|=3m$, and a collection $\mathcal{C}=\{c_1,\dots,c_{\tau}\}$ of $\tau$ $3$-element subsets of $D$, where $\tau\ge m$. For each element $c_i \in \mathcal{C}$, define the column vector $g_i \in \mathbb{R}^{3m}$ to encode which elements of $D$ are contained in $c_i$. In other words, for $i \in \{1, 2, \ldots, \tau\}$ and $j \in \{1, 2, \ldots, 3m\}$, $(g_i)_j = 1$ if element $j$ of set $D$ is in $c_i$, and $(g_i)_j = 0$ otherwise. Define the matrix $G=\left[\begin{matrix}g_1 & \cdots & g_{\tau}\end{matrix}\right]$. Furthermore, define $d=[1\ \cdots\ 1]^T\in\mathbb{R}^{3m}$. Thus $Gx=d$ has a solution $x\in\{0,1\}^{\tau}$ such that $x$ has $m$ nonzero entries if and only if the answer to the instance of $X3C$ is ``yes'' \cite{natarajan1995sparse}. 

Given the above instance of $X3C$, we then construct an instance of KFSS as follows. We define the system dynamics matrix as $A=\textrm{diag}(\lambda_1,0,\dots,0) \in\mathbb{R}^{(3m+1)\times(3m+1)}$, where $0<|\lambda_1|<1$.\footnote{We take $\lambda_1=\frac{1}{2}$ for the proof.} The set $\mathcal{Q}$ is defined to contain $\tau+1$ sensors with the collective measurement matrix 
\begin{equation}
C=\begin{bmatrix}
1 & d^T\\
\mathbf{0} & G^T
\end{bmatrix},
\label{eqn: C matrix in theorem 1}
\end{equation}
where $G$ and $d$ are defined, based on the given instance of $X3C$, as above.  The system noise covariance matrix is set to be $W=I_{(3m+1)\times(3m+1)}$, and the measurement noise covariance matrix is set to be $V=\mathbf{0}_{(\tau+1)\times (\tau+1)}$.  Finally, the cost vector is set as $b=[1\ \cdots\ 1]^T\in\mathbb{R}^{\tau+1}$, and the sensor selection budget is set as $B=m+1$.  Note that the sensor selection vector for this instance is denoted by $\mu\in\mathbb{R}^{\tau+1}$. For the above construction, since the only nonzero eigenvalue of $A$ is $\lambda_1$, we know from Lemma $\ref{Lemma:minimum trace of sigma}$(c) that $\sum_{i=2}^{3m+1}(\Sigma(\mu))_{ii}=\sum_{i=2}^{3m+1}W_{ii}=3m$ for all sensor selections $\mu$. 

We claim that the solution $\mu^*$ to the constructed instance of the KFSS problem satisfies $\textrm{trace}(\Sigma(\mu^*))=\textrm{trace}(W)=3m+1$ if and only if the answer to the given instance of the $X3C$ problem is ``yes''.

Suppose that the answer to the instance of the $X3C$ problem is ``yes''. Then $Gx=d$ has a solution such that $x$ has $m$ nonzero entries. Denote the solution as $x^*$ and denote $\textrm{supp}(x^*)=\{i_1,\dots,i_m\}$. Define $\tilde{\mu}$ as the sensor selection vector that indicates selecting the first and the $(i_1+1)$th to the $(i_m+1)$th sensors, i.e., sensors that correspond to rows $C_1$, $C_{i_1+1},\dots, C_{i_m+1}$ from \eqref{eqn: C matrix in theorem 1}. Since $Gx^*=d$, we have $[1\ -x^{*T}]C=\mathbf{e}_1$ for $C$ as defined in Eq. $(\ref{eqn: C matrix in theorem 1})$. Noting that $\textrm{supp}(x^*)=\{i_1,\dots,i_m\}$, it then follows that $\mathbf{e}_1\in\textrm{rowspace}(C(\tilde{\mu}))$. Hence, we know from Lemma \ref{Lemma:minimum trace of sigma}(a) and Lemma \ref{Lemma:minimum trace of sigma}(e) that $(\Sigma(\tilde{\mu}))_{11}=1$, which is also the minimum value of $(\Sigma(\mu))_{11}$ among all possible sensor selections $\mu$. Since $\sum_{i=2}^{3m+1}(\Sigma(\mu))_{ii}=3m$ always holds as argued above, we have $\textrm{trace}(\Sigma(\tilde{\mu}))=\textrm{trace}(W)=3m+1$ and $\tilde{\mu}$ is the optimal sensor selection, i.e., $\tilde{\mu}=\mu^*$. 

Conversely, suppose that the answer to the $X3C$ problem is ``no''. Then, for any union of $l\le m$ ($l\in\mathbb{Z}$) subsets in $\mathcal{C}$, denoted as $\mathcal{C}_l$, there exist $\omega\ge1$ $(\omega\in\mathbb{Z})$ elements in $D$ that are not covered by $\mathcal{C}_l$, i.e., for any $l\le m$ and $\mathcal{L}\triangleq\{i_1,\dots,i_l\}\subseteq\{1,\dots,\tau\}$, $G_{\mathcal{L}}\triangleq\left[\begin{matrix} g_{i_1}\ \cdots\ g_{i_l}\end{matrix}\right]$ has $\omega$ zero rows, for some $\omega\ge1$. We then show that $\textrm{trace}(\Sigma(\mu))> 3m+1$ for all sensor selections $\mu$ that satisfy the budget constraint. First, for any possible sensor selection $\mu$ that does not select the first sensor, we have the first column of $C(\mu)$ is zero (from the form of $C$ as defined in Eq. $(\ref{eqn: C matrix in theorem 1})$) and we know from Lemma $\ref{Lemma:minimum trace of sigma}$(d) that  $(\Sigma(\mu))_{11}=\frac{1}{1-\lambda_1^2}=\frac{4}{3}$, which implies that $\textrm{trace}(\Sigma(\mu))=3m+\frac{4}{3}>3m+1$.  Thus, consider sensor selections $\mu$ that select the first sensor, denote $\textrm{supp}(\mu)=\{1,i_1,\dots,i_l\}$, where $l\le m$ and define $G(\mu)=\left[\begin{matrix}g_{i_1-1} & \cdots & g_{i_l-1}\end{matrix}\right]$. We then have
\begin{equation}
C(\mu)=\begin{bmatrix}
1 & d^T\\
\mathbf{0} & G(\mu)^T
\end{bmatrix},
\label{eqn:C_mu_thm_1}
\end{equation}
where $G(\mu)^T$ has $\omega$ zero columns, for some $\omega\ge1$. As argued in Lemma $\ref{lemma:X3C no solution}$ in the appendix, there exists an orthogonal matrix $T\in\mathbb{R}^{(3m+1)\times (3m+1)}$ of the form $T = \left[\begin{smallmatrix}1 & 0 \\ 0 & N\end{smallmatrix}\right]$ (where $N$ is also an orthogonal matrix) such that 
\begin{equation*}
\tilde{C}(\mu)\triangleq C(\mu)T=\begin{bmatrix}
1 & \mathbf{\gamma} &\mathbf{\beta}\\
\mathbf{0} & \mathbf{0} & \tilde{G}(\mu)^T
\end{bmatrix}.
\end{equation*}
In the above expression, $\tilde{G}(\mu)^T\in\mathbb{R}^{l\times r}$ is of full column rank, where $r=\textrm{rank}(G(\mu)^T)$. Furthermore, $\gamma\in\mathbb{R}^{1\times(3m-r)}$ and $\omega$ of its elements are $1$'s, and $\beta\in\mathbb{R}^{1\times r}$. We perform a similarity transformation on the system with the matrix $T$ (which does not affect the trace of the error covariance matrix in general and does not change $A$ and $W$ in this case), and perform additional elementary row operations to transform $\tilde{C}(\mu)$ into the matrix
\begin{equation}
\tilde{C}'(\mu)=\begin{bmatrix}
1 & \mathbf{\gamma} &\mathbf{0}\\
\mathbf{0} & \mathbf{0} & \tilde{G}(\mu)^T
\end{bmatrix}.
\label{eqn:C matrix after row and column operations}
\end{equation}
Since $A$ and $W$ are both diagonal, and $V=\mathbf{0}$, we can obtain from Eq. $(\ref{eqn:DARE})$ that the steady state error covariance $\tilde{\Sigma}'(\mu)$ corresponding to the sensing matrix $\tilde{C}'(\mu)$ is of the form
\begin{equation*}
\tilde{\Sigma}'(\mu)=\begin{bmatrix}
\tilde{\Sigma}'_1(\mu) & \mathbf{0}\\
\mathbf{0} & \tilde{\Sigma}'_2(\mu)
\end{bmatrix},
\end{equation*}
where $\tilde{\Sigma}'_1(\mu)\in\mathbb{R}^{(3m+1-r)\times (3m+1-r)}$, denoted as $\Sigma$ for simplicity, satisfies
\begin{equation*}
\Sigma=A_1\Sigma A_1^T+W_1-A_1\Sigma C_1^T\big(C_1\Sigma C_1^T\big)^{-1}C_1\Sigma A_1^T,
\end{equation*}
where $A_1=\textrm{diag}(\lambda_1,0,\dots,0)\in\mathbb{R}^{(3m+1-r)\times(3m+1-r)}$, $C_1=[1\ \gamma]$ and $W_1=I_{(3m+1-r)\times(3m+1-r)}$. We then know from Lemma $\ref{lemma:estimation error of state 1}$ that
$(\Sigma(\mu))_{11}=(\tilde{\Sigma}'(\mu))_{11}>1$ since $\lVert \gamma \rVert_2^2\ge\omega\ge1>0$. Hence, we have $\textrm{trace}(\Sigma(\mu))> 3m+1$.

This completes the proof of the claim above. Suppose that there is an algorithm $\mathcal{A}$ that outputs the optimal solution $\mu^*$ to the instance of the KFSS problem defined above. We can call algorithm $\mathcal{A}$ to solve the $X3C$ problem. Specifically, if the algorithm $\mathcal{A}$ outputs a solution $\mu^*$ such that $\textrm{trace}(\Sigma(\mu^*))=\textrm{trace}(W)$, then the answer to the instance of $X3C$ is ``yes''; otherwise, the answer is ``no''.

Hence, we have a reduction from $X3C$ to the KFSS problem. Since $X3C$ is NP-complete and KFSS $\notin NP$, we conclude that the KFSS problem is NP-hard.
\end{proof}

\subsection{Inapproximability of the KFSS Problem}

In this section, we analyze the achievable performance of algorithms for the KFSS problem.  Specifically, consider any given instance of the KFSS problem.  For any given algorithm $\mathcal{A}$, we define the following ratio:
\begin{equation}
r_{\mathcal{A}}(\Sigma)\triangleq \frac{\textrm{trace}(\Sigma_{\mathcal{A}})}{\textrm{trace}(\Sigma_{opt})},
\label{eqn:performance ratio of algorithm}
\end{equation}
where $\Sigma_{opt}$ is the optimal solution to the KFSS problem and $\Sigma_{\mathcal{A}}$ is the solution to the KFSS problem given by algorithm $\mathcal{A}$.

In \cite{zhang2017sensor}, the authors showed that there is an upper bound for $r_{\mathcal{A}}(\Sigma)$ for any sensor selection algorithm $\mathcal{A}$, in terms of the system matrices.  However, the question of whether it is possible to find an algorithm $\mathcal{A}$ that is guaranteed to provide an approximation ratio $r_{\mathcal{A}}(\Sigma)$ that is {\it independent} of the system parameters has remained open up to this point.  In particular, it is typically desirable to find {\it constant-factor} approximation algorithms, where the ratio $r_{\mathcal{A}}(\Sigma)$ is upper-bounded by some (system-independent) constant.  Here, we provide a strong negative result and show that for the KFSS problem, there is no constant-factor approximation algorithm in general, i.e., for all polynomial-time algorithms $\mathcal{A}$ and $\forall K\in\mathbb{R}_{\ge 1}$, there are instances of the KFSS problem where $r_{\mathcal{A}}(\Sigma)>K$.

\begin{theorem}
If $P\neq NP$, then there is no polynomial-time constant-factor approximation algorithm for the KFSS problem.
\label{thm:inapprox}
\end{theorem}

\begin{proof}
Suppose that there exists such a (polynomial-time) approximation algorithm $\mathcal{A}$, i.e., $\exists K\in\mathbb{R}_{\ge 1}$ such that $r_{\mathcal{A}}(\Sigma)\le K$ for all instances of the KFSS problem, where $r_{\mathcal{A}}(\Sigma)$ is as defined in Eq. $\eqref{eqn:performance ratio of algorithm}$. We will show that $\mathcal{A}$ can be used to solve the $X3C$ problem as described in Lemma \ref{lemma:X3C}. Given an arbitrary instance of the $X3C$ problem (with a base set $D$ containing $3m$ elements and a collection $\mathcal{C}$ of $3$-element subsets of $D$), we  construct a corresponding instance of the KFSS problem in a similar way to that described in the proof of Theorem \ref{thm: KFSS is NP-hard}. Specifically, the system dynamics matrix is set as $A=\textrm{diag}(\lambda_1,0,\dots,0)\in\mathbb{R}^{(3m+1)\times(3m+1)}$, where $\lambda_1=\frac{K(3m+1)-3m-1/2}{K(3m+1)-3m}$.  The set $\mathcal{Q}$ contains $\tau+1$ sensors with collective measurement matrix
\begin{equation}
C=\begin{bmatrix}
1 & \varepsilon d^T\\
\mathbf{0} & G^T
\end{bmatrix},
\label{eqn: C matrix inapprox}
\end{equation}
where $G$, $d$ depend on the given instance of $X3C$ and are as defined in the proof of Theorem $\ref{thm: KFSS is NP-hard}$.  The constant $\varepsilon$ is chosen as $\varepsilon=2[(K(3m+1)-3m)]\Bigl\lceil\sqrt{K(3m+1)-3m-1}\Bigr\rceil+1$. The system noise covariance matrix is set to $W=I_{(3m+1)\times(3m+1)}$, and the measurement noise covariance matrix is set to be $V=\mathbf{0}_{(\tau+1)\times (\tau+1)}$. The sensor cost vector is set as $b=[1\ \cdots\ 1]^T\in\mathbb{R}^{\tau+1}$, and the sensor selection budget is set as $B=m+1$.  Note that the sensor selection vector is given by $\mu\in\mathbb{R}^{\tau+1}$. 

We claim that there exists a sensor selection vector $\mu$ such that $\textrm{trace}(\Sigma(\mu))\le K(3m+1)$ if and only if the answer to the $X3C$ problem is ``yes''. 

Suppose that the answer to the $X3C$ problem is ``yes''. We know from Theorem $\ref{thm: KFSS is NP-hard}$ that there exists a sensor selection $\mu$ such that $\textrm{trace}(\Sigma(\mu))=3m+1 \le K(3m+1)$.

Conversely, suppose that the answer to the $X3C$ problem is ``no''. Then, for any union of $l\le m$ ($l\in\mathbb{Z}$) subsets in $\mathcal{C}$, denoted as $\mathcal{C}_l$, there exist $\omega\ge1$ $(\omega\in\mathbb{Z})$ elements in $D$ that are not covered by $\mathcal{C}_l$. We follow the discussion in Theorem $\ref{thm: KFSS is NP-hard}$.  First, for any sensor selection $\mu$ that does not select the first sensor, we have $(\Sigma(\mu))_{11}=\frac{1}{1-\lambda_1^2}$. Hence, by our choice of $\lambda_1$, we have $(\Sigma(\mu))_{11}>K(3m+1)-3m$, which implies $\textrm{trace}(\Sigma(\mu))>K(3m+1)$ since $\displaystyle\sum_{i = 2}^{3m+1}(\Sigma(\mu))_{ii} = 3m$ for all possible sensor selections. Thus, consider sensor selections $\mu$ that include the first sensor. As argued in the proof of Theorem~\ref{thm: KFSS is NP-hard} leading up to Eq. $(\ref{eqn:C matrix after row and column operations})$, we can perform an orthogonal similarity transformation on the system, along with elementary row operations on the measurement matrix $C(\mu)$ to obtain a measurement matrix of the form 
\begin{equation}
\tilde{C}'(\mu)=\begin{bmatrix}
1 & \varepsilon\mathbf{\gamma} &\mathbf{0}\\
\mathbf{0} & \mathbf{0} & \tilde{G}(\mu)^T
\end{bmatrix},
\label{eqn: C matrix inapprox after row operations}
\end{equation}
where $\omega\ge1$ elements of $\gamma\in\mathbb{R}^{3m-r}$ are $1$'s and $r=\textrm{rank}(\tilde{G}(\mu)^T)$. Then, we have $\alpha^2\triangleq\varepsilon^2\lVert\gamma\rVert_2^2\ge\omega\varepsilon^2\ge\varepsilon^2$. Moreover, we obtain from Lemma $\ref{lemma:estimation error of state 1}$
\begin{equation*}
(\Sigma(\mu))_{11}=\frac{1+\alpha^2\lambda_1^2-\alpha^2+\sqrt[]{(\alpha^2-\alpha^2\lambda_1^2-1)^2+4\alpha^2}}{2}.
\end{equation*}
If we view $(\Sigma(\mu))_{11}$ as a function of $\alpha^2$, denoted as $\Sigma_{11}(\alpha^2)$, we know from  Lemma $\ref{lemma:estimation error of state 1}$ that $\Sigma_{11}(\alpha^2)$ is an increasing function of $\alpha^2$. Hence, we have $\Sigma_{11}(\alpha^2)\ge\Sigma_{11}(\varepsilon^2)$, i.e.,
\begin{equation*}
\Sigma_{11}(\alpha^2)\ge\frac{1+\varepsilon^2\lambda_1^2-\varepsilon^2+\sqrt[]{(\varepsilon^2-\varepsilon^2\lambda_1^2-1)^2+4\varepsilon^2}}{2}.
\end{equation*}
By our choices of $\lambda_1$ and $\varepsilon$, we have $(\Sigma(\mu))_{11}>K(3m+1)-3m$, which implies $\textrm{trace}(\Sigma(\mu))> K(3m+1)$.

This completes the proof of the claim above. Hence, if algorithm $\mathcal{A}$ for the KFSS problem has $r_{\mathcal{A}}(\Sigma)\le K$ for all instances, it is clear that $\mathcal{A}$ can be used to solve the $X3C$ problem by applying it to the above instance.  Specifically, if the answer to the $X3C$ instance is ``yes'', then the optimal sensor selection $\mu^*$ would yield a trace of $\Sigma(\mu^*) = 3m+1$, and thus the algorithm $\mathcal{A}$ would yield a trace no larger than $K(3m+1)$.  On the other hand, if the answer to the $X3C$ instance is ``no'', all sensor selections would yield a trace larger than $K(3m+1)$, and thus so would the sensor selection provided by $\mathcal{A}$.  In either case, the solution provided by $\mathcal{A}$ could be used to find the answer to the given $X3C$ instance.  Since $X3C$ is NP-complete, there is no polynomial-time algorithm for it if $P\neq NP$, and we get a contradiction. This completes the proof of the theorem. 
\end{proof}


\section{Greedy Algorithm}\label{sec:greedy example}
Our result in Theorem~\ref{thm:inapprox} indicates that no polynomial-time algorithm can be guaranteed to yield a solution that is within any constant factor of the optimal solution.  In particular, this result applies to the greedy algorithms that are often studied for sensor selection in the literature \cite{zhang2017sensor}, where sensors are iteratively selected in order to produce the greatest decrease in the error covariance at each iteration.  In particular, it was shown via simulations in \cite{zhang2017sensor} that such algorithms work well in practice (e.g., for randomly generated systems).  In this section, we provide an explicit example showing that greedy algorithms for KFSS can perform arbitrarily poorly, even for small systems (containing only three states).  We will focus on  the simple greedy algorithm for the KFSS problem defined as Algorithm \ref{algorithm: greedy KFSS}, for instances where all sensor costs are equal to $1$, and the sensor budget $B=p_s$ for some $p_s\in\{1,\dots,q\}$ (i.e., up to $p_s$ sensors can be chosen).  For any such instance of the KFSS problem, define $r_{gre}(\Sigma)=\frac{\textrm{trace}(\Sigma_{gre})}{\textrm{trace}(\Sigma_{opt})}$, where $\Sigma_{gre}$ is the solution of the DARE corresponding to the sensors selected by Algorithm \ref{algorithm: greedy KFSS}.

\begin{algorithm}
\textbf{Input:} System dynamics matrix $A$, set of all candidate sensors $\mathcal{Q}$, noise covariances $W$ and $V$, budget $p_s$ \\
\textbf{Output:} A set $\mathcal{S}$ of selected sensors
\caption{Greedy Algorithm for KFSS}\label{algorithm: greedy KFSS}
\begin{algorithmic}[1]

\State $k\gets 1$, $\mathcal{S}\gets\emptyset$
\For{$k\le p_s$}
  \For{$i\in\mathcal{Q}\cap \overline{\mathcal{S}}$}
    \State Calculate $\textrm{trace}(\Sigma(\mathcal{S}\cup\{i\}))$
  \EndFor
    \State $j=\mathop{\arg\min}_i\textrm{trace}(\Sigma(\mathcal{S}\cup\{i\}))$
    \State $\mathcal{S}\gets \mathcal{S}\cup\{j\}$, $k\gets k+1$
\EndFor
\end{algorithmic}
\end{algorithm}

\begin{example}
Consider an instance of the KFSS problem with matrices $W=I_{3\times 3}$, $V=\mathbf{0}_{3\times 3}$, and $A$, $C$ defined as
\begin{equation*}
\label{eq:example matrices KFSS}
  A=
  \begin{bmatrix}
    \lambda_1 & 0 & 0 \\
            0 & 0 & 0\\
            0 & 0 & 0
  \end{bmatrix},
  C=
  \begin{bmatrix}
    1 & h & h \\
    1 & 0 & h \\
    0 & 1 & 1
  \end{bmatrix},\\
\end{equation*}
where $0<|\lambda_1|<1$, $\lambda_1\in\mathbb{R}$ and $h\in\mathbb{R}_{>0}$. In addition, we have the set of candidate sensors $\mathcal{Q}=\{1,2,3\}$, the selection budget $B=2$ and the cost vector $b=[1\ 1\ 1]^T$. 
\label{exp:greedy bad results KFSS}
\end{example}

\begin{theorem}
For the instance of the KFSS problem defined in Example \ref{exp:greedy bad results KFSS}, the ratio $r_{gre}(\Sigma)=\frac{\textrm{trace}(\Sigma_{gre})}{\textrm{trace}(\Sigma_{opt})}$ satisfies
\begin{equation}
\mathop{\lim}_{h\to\infty}r_{gre}(\Sigma)= \frac{2}{3}+\frac{1}{3(1-\lambda_1^2)}.
\label{eq:limit ratio KFSS}
\end{equation}
\label{thm: ratio of special system KFSS}
\end{theorem}

\begin{proof}
We first prove that the greedy algorithm defined as Algorithm \ref{algorithm: greedy KFSS} selects sensor $2$ and sensor $3$ in its first and second iterations. Since the only nonzero eigenvalue of $A$ is $\lambda_1$, we know from Lemma $\ref{Lemma:minimum trace of sigma}$(c) that $(\Sigma(\mu))_{22}=1$ and $(\Sigma(\mu))_{33}=1$, $\forall\mu$, which implies that $(\Sigma_{gre})_{22}=1$ and $(\Sigma_{gre})_{33}=1$. Hence, we focus on determining $(\Sigma_{gre})_{11}$.

In the first iteration of the greedy algorithm, suppose the first sensor, i.e., sensor corresponding to $C_1=[1\ h\ h]$, is selected.  Then, using the result in Lemma $\ref{lemma:estimation error of state 1}$, we obtain $(\Sigma(\mu))_{11}\big|_{\mu=[1\ 0\ 0]^T}$, denoted as $\sigma_1$, to be
\begin{equation*}
\sigma_1=\frac{2}{\sqrt[]{(1-\lambda_1^2-\frac{1}{2h^2})^2+\frac{2}{h^2}}+1-\lambda_1^2-\frac{1}{2h^2}}.
\label{eqn:solution for sigma1}
\end{equation*}
Similarly, if the second sensor, i.e., the sensor corresponding to $C_2=[1\ 0\ h]$, is selected in the first iteration of the greedy algorithm, we have $(\Sigma(\mu))_{11}\big|_{\mu=[0\ 1\ 0]^T}$, denoted as $\sigma_2$, to be 
\begin{equation*}
\sigma_2=\frac{2}{\sqrt[]{(1-\lambda_1^2-\frac{1}{h^2})^2+\frac{4}{h^2}}+1-\lambda_1^2-\frac{1}{h^2}}.
\label{eqn:solution for sigma2}
\end{equation*}
If the third sensor, i.e., the sensor corresponding to $C_3=[0\ 1\ 1]$, is selected in the first iteration of the greedy algorithm, the first column of $C(\mu)\big|_{\mu=[0\ 0\ 1]^T}$ is zero, which implies $\sigma_3 \triangleq (\Sigma(\mu))_{11}\big|_{\mu=[0\ 0\ 1]^T}=\frac{1}{1-\lambda_1^2}$ based on Lemma $\ref{Lemma:minimum trace of sigma}$(d).  If we view $\sigma_2$ as a function of $h^2$, denoted as $\sigma(h^2)$, we have $\sigma_1=\sigma(2h^2)$. Since we know from Lemma $\ref{lemma:estimation error of state 1}$ that $\sigma(h^2)$ is an increasing function of $h^2$ and upper bounded by $\frac{1}{1-\lambda_1^2}$, we obtain $\sigma_2<\sigma_1<\sigma_3$, which implies that the greedy algorithm selects the second sensor in its first iteration.

In the second iteration of the greedy algorithm, if the first sensor is selected, we have $C(\mu))\big|_{\mu=[1\ 1\ 0]^T}=
\begin{bmatrix}
    1 & h & h \\
    1 & 0 & h \\
\end{bmatrix}$, on which we perform elementary row operations and obtain $\tilde{C}(\mu)\big|_{\mu=[1\ 1\ 0]^T}=
\begin{bmatrix}
    0 & h & 0 \\
    1 & 0 & h \\
\end{bmatrix}$. By direct computation from Eq. $\eqref{eqn:DARE}$, we obtain $\sigma_{12} \triangleq (\Sigma(\mu))_{11}\big|_{\mu=[1\ 1\ 0]^T}=\sigma_2$. If the third sensor is selected, we have $C(\mu)\big|_{\mu=[0\ 1\ 1]^T}=
\begin{bmatrix}
    1 & 0 & h \\
    0 & 1 & 1 \\
\end{bmatrix}$. By direct computation from Eq. $\eqref{eqn:DARE}$, we obtain $(\Sigma(\mu))_{11}\big|_{\mu=[0\ 1\ 1]^T}$, denoted as $\sigma_{23}$, to be 
\begin{equation*}
\sigma_{23}=\frac{2}{\sqrt[]{(1-\lambda_1^2-\frac{2}{h^2})^2+\frac{8}{h^2}}+1-\lambda_1^2-\frac{2}{h^2}}.
\label{eqn:solution for sigma23}
\end{equation*}
Similar to the argument above, we have $\sigma_{12}=\sigma(h^2)$ and $\sigma_{23}=\sigma(\frac{h^2}{2})$, where $\sigma(\frac{h^2}{2})<\sigma(h^2)$, which implies that the greedy algorithm selects the third sensor in its second iteration. Hence, we have $\textrm{trace}(\Sigma_{gre})=\sigma_{23}+2$.

If $\mu=[1\ 0\ 1]^T$, then $\mathbf{e}_1\in\textrm{rowspace}(C(\mu))$ and thus we know from Lemma $\ref{Lemma:minimum trace of sigma}$(a) and Lemma $\ref{Lemma:minimum trace of sigma}$(e) that $\textrm{trace}(\Sigma(\mu))=3=\textrm{trace}(W)$, which is also the minimum value of $\textrm{trace}(\Sigma(\mu))$ among all possible sensor selections $\mu$. Combining the results above and taking the limit as $h\to\infty$, we obtain the result in Eq. $(\ref{eq:limit ratio KFSS})$.
\end{proof}

Examining Eq. \eqref{eq:limit ratio KFSS}, we see that for the given instance of KFSS, we have $r_{gre}(\Sigma) \rightarrow \infty$ as $h \rightarrow \infty$ and $\lambda_1 \rightarrow 1$.  Thus, $r_{gre}(\Sigma)$ can be made arbitrarily large by choosing the parameters in the instance appropriately.  It is also useful to note that the above behavior holds for any algorithm that outputs a sensor selection that contains sensor $2$ for the above example.

\section{Conclusions}\label{sec:conclusion}

In this paper, we studied the KFSS problem for linear dynamical systems. We showed that this problem is NP-hard and has no constant-factor approximation algorithms, even under the assumption that the system is stable and each sensor has identical cost. We provided an explicit example showing how a greedy algorithm can perform arbitrarily poorly on this problem, even when the system only has three states. Our results shed new insights into the problem of sensor selection for Kalman filtering and show, in particular, that this problem is more difficult than other variants of the sensor selection problem that have submodular cost functions.  Future work on characterizing achievable (non-constant) approximation ratios, or identifying classes of systems that admit near-optimal approximation algorithms, would be of interest.




\section*{Appendix}

\subsection*{Proof of Lemma \ref{Lemma:minimum trace of sigma}:}
Since $A$ and $W$ are diagonal, the system represents a set of $n$ scalar subsystems of the form 
\begin{equation*}
x_i[k+1]=\lambda_i x_i[k] + w_i[k], \forall i\in\{1,\dots,n\},
\label{eqn:decoupled system dynamics}
\end{equation*}
where $x_i[k]$ is the $i$th state of $x[k]$ and $w_i[k]$ is a zero-mean Gaussian noise process with variance $\sigma_{w_i}^2=W_{ii}$. As $A$ is stable, the pair $(A,C(\mu))$ is detectable and the pair $(A,W^{\frac{1}{2}})$ is stabilizable for all sensor selections $\mu$. Thus, the limit $\displaystyle\mathop{\lim}_{k\to\infty}(\Sigma_{k|k-1}(\mu))_{ii}$ exists $\forall i$ (based on Lemma \ref{lemma:Anderson optimal filtering}), and is denoted as $(\Sigma(\mu))_{ii}$.

Proof of (a). Since $A$ and $W$ are diagonal, we know from Eq. $(\ref{eqn:couple priori and posteriori})$ that 
\begin{equation*}
(\Sigma(\mu))_{ii}=\lambda_i^2(\Sigma^*(\mu))_{ii}+W_{ii},
\label{eqn:coupled priori and posteriori diagonal case}
\end{equation*}
which implies $(\Sigma(\mu))_{ii}\ge W_{ii}$, $\forall i\in\{1,\dots,n\}$. Moreover, it is easy to see that $(\Sigma(\mu))_{ii}\le (\Sigma(\mathbf{0}))_{ii}$, $\forall i\in\{1,\dots,n\}$. Since $C(\mathbf{0})=\mathbf{0}$, we obtain from Eq. $(\ref{eqn:DARE})$
\begin{equation}
\Sigma(\mathbf{0})=A\Sigma(\mathbf{0})A^T+W.
\end{equation}
which implies that $(\Sigma(\mathbf{0}))_{ii}=\frac{W_{ii}}{1-\lambda_i^2}$ since $A$ is diagonal. Hence, $(\Sigma(\mu))_{ii}\le\frac{W_{ii}}{1-\lambda_i^2}$, $\forall i\in\{1,\dots,n\}$.

Proof of (b). Letting $W_{ii}=0$ in inequality $(\ref{eqn:bound for priori sigma})$, we obtain $(\Sigma(\mu))_{ii}=0$.

Proof of (c). Letting $\lambda_i=0$ in inequality $(\ref{eqn:bound for priori sigma})$, we obtain  $(\Sigma(\mu))_{ii}=W_{ii}$.

Proof of (d). Assume without loss of generality that the first column of $C(\mu)$ is zero, since we can simply renumber the states to make this the case without affecting the trace of the error covariance matrix. Hence, we have  $C(\mu)$ of the form
\begin{equation*}
C(\mu)=\begin{bmatrix}
\mathbf{0} & C_1(\mu)
\end{bmatrix}.
 \end{equation*}
Moreover, since $A$ and $W$ are diagonal and $V=\mathbf{0}$, we can obtain from Eq. $(\ref{eqn:DARE})$ that $\Sigma(\mu)$ is of the form
\begin{equation*}
\Sigma(\mu)=\begin{bmatrix}
\Sigma_1(\mu) & \mathbf{0}\\
\mathbf{0} & \Sigma_2(\mu)
\end{bmatrix},
\end{equation*}
where $\Sigma_1(\mu)=(\Sigma(\mu))_{11}$ and satisfies 
\begin{equation*}
(\Sigma(\mu))_{11}=\lambda_i^2(\Sigma(\mu))_{11}+W_{11}, 
\end{equation*}
which implies $(\Sigma(\mu))_{11}=\frac{W_{11}}{1-\lambda_1^2}$.

Proof of (e). Similar to the proof of (d), we can assume without loss of generality that $\mathbf{e}_{1}\in\textrm{rowspace}(C(\mu))$. If we further perform elementary row operations on $C(\mu)$,\footnote{Note that since we assume $V=\mathbf{0}$, it is easy to see that the Kalman filter gives the same results if we perform any elementary row operation on $C(\mu)$.} we get a matrix $\tilde{C}(\mu)$ of the form 
\begin{equation*}
\tilde{C}(\mu)=\begin{bmatrix}
1 & \mathbf{0}\\
\mathbf{0} & \tilde{C}_1(\mu)
\end{bmatrix}.
 \end{equation*}
Moreover, since $A$ and $W$ are diagonal and $V=\mathbf{0}$, we can obtain from Eq. $(\ref{eqn:DARE})$ that $\Sigma(\mu)$ is of the form
\begin{equation*}
\Sigma(\mu)=\begin{bmatrix}
\Sigma_1(\mu) & \mathbf{0}\\
\mathbf{0} & \Sigma_2(\mu)
\end{bmatrix},
\end{equation*}
where $\Sigma_1(\mu)=(\Sigma(\mu))_{11}$ and satisfies
\begin{equation*}
(\Sigma(\mu))_{11}=\lambda_1^2(\Sigma(\mu))_{11}+W_{11}-\lambda_1^2(\Sigma(\mu))_{11},
\end{equation*}
which implies $(\Sigma(\mu))_{11}=W_{11}$.\hfill\QED

\subsection*{Proof of Lemma \ref{lemma:estimation error of state 1}:}

Since $A=\textrm{diag}(\lambda_1,0,\dots,0)$ and $W=I_{n\times n}$, we obtain $x_i[k+1]=w_i[k]$, $\forall i\in\{2,\dots,n\}$ from Eq. $(\ref{eqn:system dynamics})$, where $w_i[k]'s$ are uncorrelated white Gaussian noise sequences with $\mathbb{E}[(w_i[k])^2]=1$. Hence, we have from Eq. $(\ref{eqn:all sensors measurements})$
\begin{equation*}
y[k]=[1\ \mathbf{0}_{1\times n-1}]x[k]+v'[k]=x_1[k]+v'[k],
\end{equation*}
where $v'[k]\triangleq\displaystyle\sum_{i=1}^{n-1}\gamma_i w_{i+1}[k-1]$, which is a white Gaussian noise process with $\mathbb{E}[(v'[k])^2]=\lVert\gamma\rVert_2^2$, denoted as $\alpha^2$. Hence, to compute the MSEE of state $1$ of the Kalman filter, i.e., $\Sigma_{11}$, we can consider a scalar discrete-time linear system with $A=\lambda_1$, $C=1$, $W=1$ and $V=\alpha^2$, and obtain from Eq. $(\ref{eqn:DARE})$, the scalar DARE 
\begin{equation}
\Sigma_{11}=\lambda_1^2(1-\frac{\Sigma_{11}}{\alpha^2+\Sigma_{11}})\Sigma_{11}+1.
\label{eqn:scalar DARE for state 1}
\end{equation}
Solving for $\Sigma_{11}$ in Eq. $(\ref{eqn:scalar DARE for state 1})$ and omitting the negative solution we obtain Eq. $(\ref{eqn: mean square estimation error of state 1})$. 

To show that $\Sigma_{11}$ is strictly increasing of $\alpha^2\in\mathbb{R}_{\ge0}$, we can use the result of Lemma $6$ in \cite{zhang2017sensor}. For a discrete-time linear system as defined in Eq. $(\ref{eqn:system dynamics})$ and Eq. $(\ref{eqn:all sensors measurements})$, given $A=\lambda_1$ and $W=1$, suppose we have two sensors with the measurement matrices as $C_1=C_2=1$ and the variances of the measurement noise as $V_1=\alpha_1^2$ and $V_2=\alpha_2^2$ (assumed to be Gaussian noise). Then, we have the sensor information matrix, as defined earlier in this paper, of these two sensors as $R_1=\frac{1}{\alpha_1^2}$ and $R_2=\frac{1}{\alpha_2^2}$. If $\alpha_1^2>\alpha_2^2$, then we know from Lemma $6$ in \cite{zhang2017sensor} that $\Sigma_{11}(\alpha_1^2)<\Sigma_{11}(\alpha_2^2)$. Hence, $\Sigma_{11}(\alpha^2)$ is a strictly increasing function of $\alpha^2\in\mathbb{R}_{\ge0}$. For $\alpha>0$, we can rewrite Eq. $(\ref{eqn: mean square estimation error of state 1})$ as
\begin{equation}
\Sigma_{11}(\alpha^2)=\frac{2}{\sqrt[]{(1-\lambda_1^2-\frac{1}{\alpha^2})^2+\frac{4}{\alpha^2}}+1-\lambda_1^2-\frac{1}{\alpha^2}}.
\label{eqn:rewrite mean square estimation error of state 1}
\end{equation}
By letting $\alpha\to\infty$ in Eq. $(\ref{eqn:rewrite mean square estimation error of state 1})$, we obtain $\displaystyle\mathop{\lim}_{\alpha\to\infty}\Sigma_{11}(\alpha^2)=\frac{1}{1-\lambda_1^2}$.\hfill\QED

\begin{lemma}
Consider an instance of $X3C$: a finite set $D$ with $|D|=3m$, and a collection $\mathcal{C}=\{c_1,\dots,c_{\tau}\}$ of $\tau$ $3$-element subsets of $D$, where $\tau\ge m$. For each element $c_i \in \mathcal{C}$, define the column vector $g_i \in \mathbb{R}^{3m}$ to encode which elements of $D$ are contained in $c_i$, i.e., for $i \in \{1, 2, \ldots, \tau\}$ and $j \in \{1, 2, \ldots, 3m\}$, $(g_i)_j = 1$ if element $j$ of set $D$ is in $c_i$, and $(g_i)_j = 0$ otherwise. Define the matrix 
$G=\left[\begin{matrix} g_1 & \cdots & g_{\tau}\end{matrix}\right]$. For any $l\le m$ ($l\in\mathbb{Z}$) and $\mathcal{L}\triangleq\{i_1,\dots,i_l\}\subseteq\{1,\dots,\tau\}$, define $G_{\mathcal{L}}=\left[\begin{matrix} g_{i_1} & \cdots & g_{i_l}\end{matrix}\right]$ and denote $\textrm{rank}(G_{\mathcal{L}})=r_{\mathcal{L}}$.\footnote{We drop the subscript $\mathcal{L}$ on $r$ for notational simplicity.} Furthermore, define $d=[1\cdots 1]^T\in\mathbb{R}^{3m}$. If the answer to the $X3C$ problem is ``no'', then for all $\mathcal{L}$ with $|\mathcal{L}|\le m$, there exists an orthogonal matrix $N\in\mathbb{R}^{3m\times 3m}$ such that
\begin{equation}
\begin{bmatrix}
d^T \\
G_{\mathcal{L}}^T
\end{bmatrix}N=\begin{bmatrix}
\mathbf{\gamma} & \mathbf{\beta}\\
\mathbf{0} & \tilde{G}_{\mathcal{L}}^T
\end{bmatrix}, 
\label{eqn: X3C no solution}
\end{equation}
where $\tilde{G}_{\mathcal{L}}^T\in\mathbb{R}^{l\times r}$ is of full column rank, $\gamma\in\mathbb{R}^{1\times(3m-r)}$ and $\omega\ge1$ ($\omega\in\mathbb{Z}$) elements of $\mathbf{\gamma}$ are $1$'s , and $\beta\in\mathbb{R}^{1\times r}$. Further elementary row operations on $\begin{bmatrix}
\mathbf{\gamma} & \mathbf{\beta}\\
\mathbf{0} & \tilde{G}_{\mathcal{L}}^T
\end{bmatrix}$ transform it into the form $\begin{bmatrix}
\mathbf{\gamma} & \mathbf{0}\\
\mathbf{0} & \tilde{G}_{\mathcal{L}}^T
\end{bmatrix}$. 
\label{lemma:X3C no solution}
\end{lemma}

\begin{proof}
Assume without loss of generality that there are no identical subsets in $\mathcal{C}$. Since $\textrm{rank}(G_{\mathcal{L}}^T)=r$, the dimension of $\textrm{nullspace}(G_{\mathcal{L}}^T)$ is $3m-r$. We choose an orthonormal basis of $\textrm{nullspace}(G_{\mathcal{L}}^T)$ and let it form the first $3m-r$ columns of $N$, denoted as $N_1$. Then, we choose an orthonormal basis of $\textrm{rowspace}(G_{\mathcal{L}}^T)$ and let it form the rest of the $r$ columns of $N$, denoted as $N_2$. Hence, $N=\begin{bmatrix}N_1 & N_2\end{bmatrix}\in\mathbb{R}^{3m\times 3m}$ is an orthogonal matrix. Furthermore, since the answer to the $X3C$ problem is ``no'', for any union of $l\le m$ ($l\in\mathbb{Z}$) subsets in $\mathcal{C}$, denoted as $\mathcal{C}_l$, there exist $\omega\ge1$ $(\omega\in\mathbb{Z})$ elements in $D$ that are not covered by $\mathcal{C}_l$, i.e., $G_{\mathcal{L}}^T$ has $\omega$ zero columns. Denote these as the $j_1$th, $\dots$, $j_{\omega}$th columns of $G_{\mathcal{L}}^T$, where $\{j_1,\dots,j_{\omega}\}\subseteq\{1,\dots,3m\}$. Hence, we can always choose $\mathbf{e}_{j_1},\dots,\mathbf{e}_{j_{\omega}}$ to be in the orthonormal basis of $\textrm{nullspace}(G_{\mathcal{L}}^T)$, i.e., as columns of $N_1$. Constructing $N$ in this way, we have $G_{\mathcal{L}}^T N_1=\mathbf{0}$ and $G_{\mathcal{L}}^T N_2=\tilde{G}_{\mathcal{L}}^T$, where $\tilde{G}_{\mathcal{L}}^T\in\mathbb{R}^{l\times r}$ is of full column rank since the columns of $N_2$ form an orthonormal basis of $\textrm{rowspace}(G_{\mathcal{L}}^T)$ and $r\le l$. Moreover, we have $d^T N_1=\mathbf{\gamma}$ and $d^T N_2=\mathbf{\beta}$, where $\omega$ elements of $\mathbf{\gamma}$ are $1$'s (since $d^T\mathbf{e}_{j_s}^T=1$, $\forall s\in\{1,\dots,\omega\}$). Combining these results, we obtain Eq. $(\ref{eqn: X3C no solution})$. Since $\tilde{G}_{\mathcal{L}}^T\in\mathbb{R}^{l\times r}$ is of full column rank, we can perform elementary row operations on $\begin{bmatrix}
\mathbf{\gamma} & \mathbf{\beta}\\
\mathbf{0} & \tilde{G}_{\mathcal{L}}^T
\end{bmatrix}$ and obtain $\begin{bmatrix}
\mathbf{\gamma} & \mathbf{0}\\
\mathbf{0} & \tilde{G}_{\mathcal{L}}^T
\end{bmatrix}$.
\end{proof}

\bibliographystyle{IEEEtran}
\bibliography{main}

\end{document}